\documentclass[11pt,reqno]{amsart}

\usepackage{xspace}
\usepackage{amsmath,amsthm,amsfonts,amssymb,latexsym,mathrsfs,color}
\usepackage{hyperref}

\newtheorem{theorem}{Theorem}
\newtheorem{corollary}[theorem]{Corollary}
\newtheorem{proposition}[theorem]{Proposition}

\newcommand{\rss}{\mathcal{RS}}
\newcommand{\altdesb}{{\rm altdesb\,}}
\newcommand{\altasc}{{\rm altasc\,}}
\newcommand{\mbn}{{\mathcal S}^B_n}
\newcommand{\mb}{{\mathcal S}^B}

\newcommand{\lpk}{{\rm lpk\,}}

\newcommand{\des}{{\rm des\,}}
\newcommand{\altdes}{{\rm altdes\,}}

\newcommand{\msn}{\mathfrak{S}_n}
\newcommand{\ms}{\mathfrak{S}}

\newcommand{\lrf}[1]{\lfloor #1\rfloor}

\newcommand{\Eulerian}[2]{\genfrac{<}{>}{0pt}{}{#1}{#2}}

\newcommand{\arxiv}[1]{\href{http://arxiv.org/abs/#1}{\texttt{arXiv:#1}}}
\title{Alternating Eulerian polynomials and left peak polynomials}
\author[S.-M.~Ma]{Shi-Mei Ma}
\address{School of Mathematics and Statistics,
        Northeastern University at Qinhuangdao,
         Hebei 066004, P.R. China}
\email{shimeimapapers@163.com (S.-M. Ma)}
\author[Q.~Fang]{Qi Fang}
\address{School of Mathematics,
        Northeastern University, Shenyang 110004, P.R.
        China}
\email{qifangpapers@stumail.neu.edu.cn (Q. Fang)}
\author[T.~Mansour]{Toufik Mansour}
\address{Department of Mathematics, University of Haifa, 3498838 Haifa, Israel}
\email{toufik@math.haifa.ac.il (T. Mansour)}
\author[Y.-N. Yeh]{Yeong-Nan Yeh}
\address{Institute of Mathematics,
        Academia Sinica, Taipei, Taiwan}
\email{mayeh@math.sinica.edu.tw (Y.-N. Yeh)}
\begin{document}
\begin{abstract}
In this paper we present grammatical interpretations of the alternating Eulerian polynomials of types $A$ and $B$. As applications, we derive several properties of
the type $B$ alternating Eulerian polynomials, including combinatorial expansions,
recurrence relations and generating functions.
We establish an interesting connection between
alternating Eulerian polynomials of type $B$ and left peak polynomials of permutations in the symmetric group, which
implies that the type $B$ alternating Eulerian polynomials have gamma-vectors alternate in sign.
\end{abstract}


\maketitle
\section{Introduction}
Let $[n]=\{1,2,\ldots,n\}$, and let $\pm[n]=[n]\cup\{\overline{1},\ldots,\overline{n}\}$, where $\overline{i}=-i$.
Denote by $\msn$ the {\it symmetric group} of all permutations of $[n]$, and denote by $\mbn$ the {\it hyperoctahedral group} of rank $n$.
Elements of $\mbn$ are signed permutations of $\pm[n]$ with the property that $\sigma(-i)=-\sigma(i)$ for all $i\in [n]$.
Let $\pi=\pi(1)\pi(2)\cdots\pi(n)\in\msn$ and let $\sigma=\sigma(1)\sigma(2)\cdots\sigma(n)\in \mbn$.
In this paper, we always assume that signed permutations are prepended by 0. That is, we identify the signed permutation $\sigma\in\mbn$ with
the word $\sigma(0)\sigma(1)\sigma(2)\cdots\sigma(n)$, where $\sigma(0)=0$.
The numbers of descents of $\pi$ and $\sigma$ are respectively defined by
$\des(\pi)=\#\{i\in[n-1]:~\pi(i)>\pi(i+1)\}$ and
$$\des_B(\sigma)=\#\{i\in\{0,1,\ldots,n-1\}\mid \sigma(i)>\sigma({i+1})\},$$
and their enumerative polynomials are the {\it Eulerian polynomials of types $A$ and $B$}, i.e.,
$$A_n(x)=\sum_{\pi\in\msn}x^{\des(\pi)},~B_n(x)=\sum_{\sigma\in\mbn}x^{\des_B(\sigma)}.$$
In the past decades, there is a larger literature devoted to combinatorial expansions of
the Eulerian polynomials of types $A$ and $B$, see~\cite{Branden08,Chow08,Petersen15,Zeng16} and references therein.
The purpose of this paper is to show that some of the crucial properties of the Eulerian polynomials have nice analogues for the alternating Eulerian polynomials.

Following Chebikin~\cite{Chebikin08}, an index $i$ is an {\it alternating descent} if
$\pi(2i)<\pi(2i+1)$ or $\pi(2i-1)>\pi(2i)$, where $1\leqslant i\leqslant \lrf{n/2}$. Let $\altdes(\pi)$ be the number of alternating descents of $\pi$.
The {\it alternating Eulerian polynomials} are defined by
$$\widehat{A}_n(x)=\sum_{\pi\in\msn}x^{\altdes(\pi)}.$$
The first few $\widehat{A}_n(x)$ are given as follows:
\begin{align*}
   \widehat{A}_1(x)& =1,~
  \widehat{A}_2(x)=1+x,~
  \widehat{A}_3(x)=2+2x+2x^2, \\
  \widehat{A}_4(x)& =5+7x+7x^2+5x^3,~\widehat{A}_5(x)=16+26x+36x^2+26x^3+16x^4.
\end{align*}
It follows from~\cite[Theorem~4.2]{Chebikin08} that
\begin{equation*}\label{EGF-Anx}
\sum_{n=1}^\infty \widehat{A}_{n}(x)\frac{z^n}{n!}=\frac{\sec(1-x)z+\tan(1-x)z-1}{1-x(\sec(1-x)z+\tan(1-x)z)}.
\end{equation*}
Equivalently, one has
\begin{equation}\label{EGF-Anx}
1+\sum_{n=1}^\infty x\widehat{A}_{n}(x)\frac{z^n}{n!}=\frac{1-x}{1-x(\sec(1-x)z+\tan(1-x)z)}.
\end{equation}
Subsequently, Remmel~\cite{Remmel12} studied the generating function for the joint distribution of alternating descents and alternating major index.
Let $\altdesb(\sigma)$ be the number of indices $i$ of $\sigma\in\mbn$ such that $\sigma(2i)<\sigma(2i+1)$ or $\sigma(2i-1)>\sigma(2i)$, where $1\leqslant i\leqslant \lrf{n/2}$.
Remmel~\cite{Remmel12} studied the statistic $\altdesb(\sigma)$ over $\mbn$.

Gessel and Zhuang~\cite{Gessel14} reproved~\eqref{EGF-Anx} by using a homomorphism of noncommutative symmetric functions
and find the exponential generating function for permutations with all valleys
even and all peaks odd. Ma and Yeh~\cite{Ma16} derived an explicit formula of the alternating Eulerian numbers by using the
derivative polynomials for the tangent function. Let $\widehat{A}_n(x)=\sum_{k=0}^{n-1}\widehat{A}(n,k)x^k$.
Ma and Yeh~\cite{Ma16} found that
the numbers $\widehat{A}(n,k)$ satisfy the recurrence relation
\begin{equation}\label{Ank-recu}
2\widehat{A}(n+1,k)=(k+1)(\widehat{A}(n,k+1)+\widehat{A}(n,k-1))+(n-k+1)(\widehat{A}(n,k)+\widehat{A}(n,k-2)),
\end{equation}
with initial conditions $\widehat{A}(1,0)=1$ and $\widehat{A}(1,k)=0$ for $k\geqslant1$.
Very recently, Lin etal.~\cite{Lin2011} found that the alternating descent polynomials $\widehat{A}_n(x)$ are
unimodal and alternatingly $\gamma$-positive. 

As a variation on the descent statistic $\des_B$, we say that an index $i$ is an {\it alternating descent} (resp.~{\it alternating ascent}) of $\sigma\in \mbn$
if $\sigma(2i)<\sigma(2i+1)$ or $\sigma(2i+1)>\sigma(2i+2)$ (resp. $\sigma(2i)>\sigma(2i+1)$ or $\sigma(2i+1)<\sigma(2i+2)$), where $0\leqslant i\leqslant \lrf{n/2}$ and $\sigma(0)=0$.
Let $\altdes_B(\sigma)$ (resp.~$\altasc_B(\sigma)$) be the number of alternating descents (resp.~alternating ascents) of $\sigma$.
It is clear that $$\altdes_B(\sigma)=\left\{
                                      \begin{array}{ll}
                                        \altdesb(\sigma)+1, & \hbox{if $\sigma(0)<\sigma(1)$;} \\
                                         \altdesb(\sigma), & \hbox{if $\sigma(0)>\sigma(1)$.}
                                      \end{array}
                                    \right.
$$
Then the equations
$$\widehat{B}_n(x)=\sum_{\sigma\in\mbn}x^{\altdes_B(\sigma)}=\sum_{k=0}^n\widehat{B}(n,k)x^k$$
define the type $B$ {\it alternating Eulerian polynomials} and {\it alternating Eulerian numbers}.
Below are the polynomials $\widehat{B}_n(x)$ for $n\leqslant 4$:
\begin{align*}
   \widehat{B}_0(x)& =1,~
  \widehat{B}_1(x)=1+x,~
  \widehat{B}_2(x)=3+2x+3x^2, \\
  \widehat{B}_3(x)&=11+13x+13x^2+11x^3,
    \widehat{B}_4(x)=57+76x+118x^2+76x^3+57x^4.
\end{align*}

The Springer numbers are introduced by Springer~\cite{Springer71} in the study of irreducible root
system of type $B_n$.
Let $s_n$ denote the $n$-th {\it Springer number}. Springer~\cite{Springer71} derived that
\begin{equation*}\label{sn}
\sum_{n=0}^{\infty}s_n\frac{z^n}{n!}=\frac{1}{\cos(z)-\sin(z)}=1+z+3\frac{z^2}{2!}+11\frac{z^3}{3!}+57\frac{z^4}{4!}+\cdots.
\end{equation*}
Arnold~\cite{Arnold} found that $s_n$ is the number of snakes in $\mbn$, where a {\it snake} in $\mbn$ is a signed permutation $\sigma$ such that
$\sigma(2i)<\sigma(2i+1)$ and $\sigma(2i+1)>\sigma(2i+2)$, where $0\leqslant i\leqslant \lrf{(n-1)/2}$. By using Corollary~\ref{unimodal}, we see that
$\widehat{B}(n,0)=\widehat{B}(n,n)=s_n$, which also can be easily deduced by using the definition of $\altdes_B(\sigma)$.

In the next section, we present the main results of this paper.
\section{Main results}
\subsection{Grammatical interpretations}
\hspace*{\parindent}

For an alphabet $A$, let $\mathbb{Q}[[A]]$ be the rational commutative ring of formal power
series in monomials formed from letters in $A$. Following Chen~\cite{Chen93}, a {\it context-free grammar} over
$A$ is a function $G: A\rightarrow \mathbb{Q}[[A]]$ that replace a letter in $A$ by a formal function over $A$.
The formal derivative $D_G$ is a linear operator defined with respect to a context-free grammar $G$. More precisely,
the derivative $D_G$: $\mathbb{Q}[[A]]\rightarrow \mathbb{Q}[[A]]$ is defined as follows:
for $x\in A$, we have $D_G(x)=G(x)$; for a monomial $u$ in $\mathbb{Q}[[A]]$, $D_G(u)$ is defined so that $D_G$ is a derivation,
and for a general element $q\in\mathbb{Q}[[A]]$, $D_G(q)$ is defined by linearity.
Following~\cite{Chen17}, a {\it grammatical labeling} is an assignment of the underlying elements of a combinatorial structure
with variables, which is consistent with the substitution rules of a grammar.

Recall that the Eulerian numbers of types $A$ and $B$ are respectively defined as follows:
$$\Eulerian{n}{k}=\#\{\pi\in\msn: \des(\pi)=k\},~ B(n,k)=\#\{\pi\in\mbn: \des_B(\pi)=k\}.$$
Let us now recall two results on
context-free grammars.
\begin{proposition}[{\cite[Section~2.1]{Dumont96}}]\label{Dumont96}
If $A=\{x,y\}$ and $G=\{x\rightarrow xy, y\rightarrow xy\}$,
then
\begin{equation*}
D_{G}^n(x)=x\sum_{k=0}^{n-1}\Eulerian{n}{k}x^{k}y^{n-k}\quad\textrm{for $n\geqslant 1$}.
\end{equation*}
\end{proposition}

\begin{proposition}[{\cite[Theorem~10]{Ma131}}]\label{Ma13}
If $A=\{x,y\}$ and $G=\{x\rightarrow xy^2, y\rightarrow x^2y\}$,
then for $n\geqslant 1$,
$$D_{G}^n(x^2)=2^n\sum_{k=0}^{n-1}\Eulerian{n}{k}x^{2n-2k}y^{2k+2},~D_{G}^n(xy)=\sum_{k=0}^nB(n,k)x^{2n-2k+1}y^{2k+1}.$$
\end{proposition}

We now give a grammatical interpretation of $\widehat{A}(n,k)$.
\begin{proposition}\label{lemma01}
If $A=\{e,x,y\}$ and $$G_1=\left\{e\rightarrow e(x+y),x\rightarrow \frac{1}{2}(x^2+y^2),y\rightarrow \frac{1}{2}(x^2+y^2)\right\},$$
then
\begin{equation}
D_{G_1}^n(e)=e\sum_{k=0}^n\widehat{A}(n+1,k)x^ky^{n-k}.
\end{equation}
\end{proposition}
\begin{proof}
Note that $D_{G_1}=e(x+y)$ and $D_{G_1}^2=e(2x^2+2xy+2y^2)$. Thus the result holds for $n=1,2$. Assume that the result holds for $n=m$.
Then
\begin{align*}
&D_{G_1}^{m+1}(e)
=D_{G_1}\left(e\sum_{k=0}^m\widehat{A}(m+1,k)x^ky^{m-k}\right)\\
&=e\sum_{k=0}^m\widehat{A}(m+1,k)\left(x^{k+1}y^{m-k}+x^ky^{m-k+1}+\frac{1}{2}(x^2+y^2)(kx^{k-1}y^{m-k}+(m-k)x^ky^{m-k-1})\right).
\end{align*}
Equating the coefficients of $x^ky^{m-k+1}$ in both sides of the above formula, we get
$$2\widehat{A}(m+2,k)=(k+1)(\widehat{A}(m+1,k+1)+\widehat{A}(m+1,k-1))+(m-k+2)(\widehat{A}(m+1,k)+\widehat{A}(m+1,k-2)).$$
Comparing this with~\eqref{Ank-recu}, we immediately get the desired result.
\end{proof}

We can now present the first main result of this paper.
\begin{theorem}\label{lemma02}
If $A=\{e,x,y\}$ and $G_2=\{e\rightarrow e(x+y),x\rightarrow x^2+y^2,y\rightarrow x^2+y^2\}$,
then
\begin{equation}\label{DGe}
D_{G_2}^n(e)=e\sum_{k=0}^n\widehat{B}(n,k)x^ky^{n-k}.
\end{equation}
\end{theorem}
\begin{proof}
We first introduce a grammatical labeling of $\sigma\in \mbn$ as follows:
\begin{itemize}
  \item [\rm ($L_1$)]If $i$ is an alternating descent, then put a superscript label $x$ right after $\sigma(i)$;
 \item [\rm ($L_2$)]If $i$ is an alternating ascent, then put a superscript label $y$ right after $\sigma(i)$;
\item [\rm ($L_3$)] Put a superscript label $e$ right after $\sigma$.
\end{itemize}
The weight of $\sigma$ is
defined as the product of its labels, i.e., $w(\sigma)=ex^{\altdes_B(\sigma)}y^{\altasc_B(\sigma)}$.
Recall that we always set $\sigma(0)=0$.
Note that $\mb_1=\{0^x1^e,0^y\overline{1}^e\}$ and
$$\mb_2=\{0^x1^y2^e,0^x1^x\overline{2}^e,0^x2^x1^e,0^y\overline{2}^y1^e,0^y\overline{1}^y2^e,0^y\overline{1}~^x\overline{2}^e,0^x2^x\overline{1}^e,0^y\overline{2}~^y\overline{1}\}^e.$$
Hence the result holds for $n=1,2$.
Suppose we get all labeled permutations in $\mb_{n-1}$, where $n\geqslant 3$. Let
$\sigma'$ be obtained from $\sigma\in \mb_{n-1}$ by inserting the entry $n$ or $\overline{n}$.
If we put $n$ or $\overline{n}$ at the end of $\sigma$, we distinguish two cases:
\begin{itemize}
  \item [\rm ($c_1$)]When $n$ is odd, the changes of labeling are illustrated as follows:
  $$\sigma\rightarrow \sigma'=\sigma^xn^e, \sigma\rightarrow \sigma'=\sigma^y\overline{n}^e;$$
 \item [\rm ($c_2$)]If $n$ is even, the changes of labeling are illustrated as follows:
   $$\sigma\rightarrow \sigma'=\sigma^yn^e, \sigma\rightarrow \sigma'=\sigma^x\overline{n}^e;$$
\end{itemize}
In each case, the insertion of $n$ or $\overline{n}$ corresponds to the substitution rule $e\rightarrow e(x+y)$.

Before we present the rest part of the proof,
we need to introduce an involution $\mathcal{C}$ (the complement) on words
with distinct letters from $\mathbb{Z}$. For a word $w=w_1w_2\cdots w_n$  of length $n$ with distinct letters from $\mathbb{Z}$, let
$\mathcal{C}(w)=w'_1w'_2\cdots w'_n$ be the  word such that if $w_i$ is the $l_i$-th largest letter in $w$ then $w'_i$ is the $l_i$-th smallest letter in $w$.
For example, if $w=3\overline{6}71\overline{2}$ then $\mathcal{C}(w)=\overline{2}7\overline{6}13$.
For each signed permutation $\sigma\in\mb_{n-1}$ and $0\leqslant j\leqslant n-2$, we call two permutations in $\mbn$,
 $$
\sigma(0)\sigma(1)\sigma(2)\cdots\sigma(j)n\mathcal{C}\left(\sigma({j+1})\sigma({j+2})\cdots\sigma(n-1)\right),$$
$$\sigma(0)\sigma(1)\sigma(2)\cdots\sigma(j)\overline{n}\mathcal{C}\left(\sigma({j+1})\sigma({j+2})\cdots\sigma(n-1)\right),
 $$
the inserting $n$ and the inserting $\overline{n}$ right after $\sigma(j)$, respectively.

We distinguish two cases to insert $n$ or $\overline{n}$ right after an alternating descent:
\begin{itemize}
  \item [\rm ($c_1$)]If we insert $n$ or $\overline{n}$ right after an even alternating descent, then the changes of labeling are illustrated as follows:
$$\cdots\sigma(2i)^x\sigma({2i+1})\cdots\mapsto \cdots\sigma(2i)^xn^x\mathcal{C}\left(\sigma({2i+1})\cdots\sigma(n-1)\right);$$
$$\cdots\sigma(2i)^x\sigma({2i+1})\cdots\mapsto \cdots\sigma(2i)^y\overline{n}^y\mathcal{C}\left(\sigma({2i+1})\cdots\sigma(n-1)\right).$$
 \item [\rm ($c_2$)]If we insert $n$ or $\overline{n}$ right after an odd alternating descent, then the changes of labeling are illustrated as follows:
$$\cdots\sigma(2i+1)^x\sigma({2i+2})\cdots\mapsto \cdots\sigma(2i+1)^yn^y\mathcal{C}\left(\sigma({2i+2})\cdots\sigma(n-1)\right);$$
$$\cdots\sigma(2i+1)^x\sigma({2i+2})\cdots\mapsto \cdots\sigma(2i+1)^x\overline{n}^x\mathcal{C}\left(\sigma({2i+2})\cdots\sigma(n-1)\right).$$
\end{itemize}
In each case, the insertion of $n$ or $\overline{n}$ corresponds to the substitution rule $x\rightarrow x^2+y^2$.

We distinguish two cases to insert $n$ or $\overline{n}$ right after an alternating ascent:
\begin{itemize}
  \item [\rm ($c_1$)]If we insert $n$ or $\overline{n}$ right after an even alternating ascent, then the changes of labeling are illustrated as follows:
$$\cdots\sigma(2i)^y\sigma({2i+1})\cdots\mapsto \cdots\sigma(2i)^xn^x\mathcal{C}\left(\sigma({2i+1})\cdots\sigma(n-1)\right);$$
$$\cdots\sigma(2i)^y\sigma({2i+1})\cdots\mapsto \cdots\sigma(2i)^y\overline{n}^y\mathcal{C}\left(\sigma({2i+1})\cdots\sigma(n-1)\right).$$
 \item [\rm ($c_2$)]If we insert $n$ or $\overline{n}$ right after an odd alternating ascent, then the changes of labeling are illustrated as follows:
$$\cdots\sigma(2i+1)^y\sigma({2i+2})\cdots\mapsto \cdots\sigma(2i+1)^yn^y\mathcal{C}\left(\sigma({2i+2})\cdots\sigma(n-1)\right);$$
$$\cdots\sigma(2i+1)^y\sigma({2i+2})\cdots\mapsto \cdots\sigma(2i+1)^x\overline{n}^x\mathcal{C}\left(\sigma({2i+2})\cdots\sigma(n-1)\right).$$
\end{itemize}
In each case, the insertion of $n$ or $\overline{n}$ corresponds to the substitution rule $y\rightarrow x^2+y^2$.

Therefore, in each case, the insertion of $n$ or $\overline{n}$ corresponds to one substitution rule in $G_2$.
It is easy to check that the action of $D_{G_2}$ on elements of $\mb_{n-1}$ generates all elements of $\mbn$. This completes the proof.
\end{proof}

It follows from~\eqref{DGe} that
\begin{align*}
&D_{G_2}^{n+1}(e)
=D_{G_2}\left(e\sum_{k=0}^n\widehat{B}(n,k)x^ky^{n-k}\right)\\
&=e\sum_{k=0}^n\widehat{B}(n,k)\left(x^{k+1}y^{n-k}+x^ky^{n-k+1}+(x^2+y^2)(kx^{k-1}y^{n-k}+(n-k)x^ky^{n-k-1})\right).
\end{align*}
Equating the coefficients of $x^ky^{n-k+1}$ in both sides, we get the following corollary.
\begin{corollary}
The numbers $\widehat{B}(n,k)$ satisfy the recurrence relation
\begin{equation}\label{Bnkrecu01}
\widehat{B}(n+1,k)=(k+1)\widehat{B}(n,k+1)+k\widehat{B}(n,k-1)+(n-k+1)\widehat{B}(n,k)+(n-k+2)\widehat{B}(n,k-2),
\end{equation}
with the initial conditions $\widehat{B}(1,0)=\widehat{B}(1,1)=1$ and $\widehat{B}(n,k)=0$ for $k>1$.
Equivalently,
\begin{equation}\label{Bnkrecu}
\widehat{B}_{n+1}(x)=(1+n+x+nx^2)\widehat{B}_{n}(x)+(1-x)(1+x^2)\frac{\mathrm{d}}{\mathrm{d}x}\widehat{B}_{n}(x),
\end{equation}
with the initial conditions $\widehat{B}_{0}(x)=1$ and $\widehat{B}_{1}(x)=1+x$.
\end{corollary}

\begin{corollary}\label{unimodal}
For $n\geqslant 3$, the polynomials $\widehat{B}_{n}(x)$ are palindromic and unimodal.
\end{corollary}
\begin{proof}
In the grammar $G_2$, the symmetry of $x$ and $y$ implies that $\widehat{{B}}_n(x)$ is palindromic.
From the initial values, we see that both $\widehat{{B}}_3(x)$ and $\widehat{{B}}_4(x)$ are unimodal.
We proceed by induction. When $n=2m$,
assume that $$\widehat{B}(n,0)<\widehat{B}(n,1)<\cdots <\widehat{B}(n,m)>\widehat{B}(n,m+1)>\cdots >\widehat{B}(n,n).$$
When $n=2m+1$,
assume that $$\widehat{B}(n,0)<\widehat{B}(n,1)<\cdots <\widehat{B}(n,m)=\widehat{B}(n,m+1)>\widehat{B}(n,m+2)>\cdots >\widehat{B}(n,n).$$
For $n\geqslant 4$ and $0\leqslant k\leqslant \lrf{n/2}$,
it follows from~\eqref{Bnkrecu01} that
\begin{align*}
&I_{n+1,k}=\widehat{B}(n+1,k)-\widehat{B}(n+1,k-1)\\
&=(k+1)\widehat{B}(n,k+1)+(n-2k+1)\widehat{B}(n,k)+(n-2k+3)\widehat{B}(n,k-2)-\\
&(n-2k+2)\widehat{B}(n,k-1)-(n-k+3)\widehat{B}(n,k-3)\\
&\geqslant \widehat{B}(n,k+1)+ k\widehat{B}(n,k-2)+(n-2k+1)\widehat{B}(n,k)+\\
&(n-2k+3)\widehat{B}(n,k-2)-(n-2k+2)\widehat{B}(n,k-1)-(n-k+3)\widehat{B}(n,k-3)\\
&=\widehat{B}(n,k+1)+(n-2k+1)\widehat{B}(n,k)+(n-k+3)\widehat{B}(n,k-2)-\\
&(n-2k+2)\widehat{B}(n,k-1)-(n-k+3)\widehat{B}(n,k-3)\\
&= \widehat{B}(n,k+1)+(n-2k+1)\widehat{B}(n,k)-(n-2k+2)\widehat{B}(n,k-1)]+\\
& (n-k+3)(\widehat{B}(n,k-2)-\widehat{B}(n,k-3)).
\end{align*}
Therefore, for any $0\leqslant k\leqslant \lrf{n/2}$, we get
$$I_{n+1,k}\geqslant\widehat{B}(n,k+1)+(n-2k+1)\widehat{B}(n,k)-(n-2k+2)\widehat{B}(n,k-1)]\geqslant 0.$$
When $n=2m+1$, we need to show that $$\widehat{B}(2m+2,m+1)> \widehat{B}(2m+2,m).$$
From the above discussion, we get
\begin{align*}
&I_{2m+2,m+1}=\widehat{B}(2m+2,m+1)-\widehat{B}(2m+2,m)\\
&\geqslant\widehat{B}(2m+1,m+2)-\widehat{B}(2m+1,m)+(m+3)\left(\widehat{B}(2m+1,m-1)-\widehat{B}(2m+1,m-2)\right)\\
&=-\left(\widehat{B}(2m+1,m)-\widehat{B}(2m+1,m-1)\right)+(m+3)\left(\widehat{B}(2m+1,m-1)-\widehat{B}(2m+1,m-2)\right).
\end{align*}
Hence we obtain
\begin{equation}\label{Inrecu}
I_{2m+2,m+1}\geqslant -I_{2m+1,m}+(m+3)I_{2m+1,m-1}.
\end{equation}

Combining~\eqref{Inrecu} and the following expression,
\begin{align*}
I_{n+1,k}&=(k+1)\widehat{B}(n,k+1)+(n-2k+1)\widehat{B}(n,k)+(n-2k+3)\widehat{B}(n,k-2)-\\
&(n-2k+2)\widehat{B}(n,k-1)-(n-k+3)\widehat{B}(n,k-3),
\end{align*}
we have
\begin{align*}
&-I_{2m+1,m}+(m+3)I_{2m+1,m-1}\\
&=2\widehat{B}(2m,m-1)+(m+3)\widehat{B}(2m,m-3)-(m+1)\widehat{B}(2m,m+1)-\\
&\widehat{B}(2m,m)-3\widehat{B}(2m,m-2)+(m+3)m\widehat{B}(2m,m)+3(m+3)\widehat{B}(2m,m-1)+\\
&5(m+3)\widehat{B}(2m,m-3)-4(m+3)\widehat{B}(2m,m-2)-(m+3)(m+4)\widehat{B}(2m,m-4).
\end{align*}
Since $\widehat{B}(2m,m+1)=\widehat{B}(2m,m-1)$, it follows that
\begin{align*}
&-I_{2m+1,m}+(m+3)I_{2m+1,m-1}\\
&=(m^2+3m-1)\widehat{B}(2m,m)+(2m+10)\widehat{B}(2m,m-1)+6(m+3)\widehat{B}(2m,m-3)-\\
&(4m+15)\widehat{B}(2m,m-2)-(m+3)(m+4)\widehat{B}(2m,m-4)\\
&\geqslant (m^2+m-6)\widehat{B}(2m,m)+6(m+3)\widehat{B}(2m,m-3)-(m+3)(m+4)\widehat{B}(2m,m-4)\\
&\geqslant (m+3)(m+4)\widehat{B}(2m,m-3)-(m+3)(m+4)\widehat{B}(2m,m-4)\geqslant 0.
\end{align*}
This completes the proof.
\end{proof}

\begin{proposition}
We have
\begin{equation}
\widehat{{{B}}}(x;z)=\sum_{n=0}^\infty \widehat{B}_{n}(x)\frac{z^n}{n!}=\frac{x-1}{(x-1)\cos(z(x-1))-(x+1)\sin(z(x-1))}.
\end{equation}
\end{proposition}
\begin{proof}
Rewrite~\eqref{Bnkrecu} in terms of generating function $\widehat{B}$, we obtain
\begin{equation}\label{EGF}
\frac{\partial}{\partial z}\widehat{B}(x;z)=(1+x)\widehat{B}(x;z)+(1+x^2)z\frac{\partial}{\partial z}\widehat{B}(x;z)+(1-x)(1+x^2)\frac{\partial}{\partial x}\widehat{B}(x;z).
\end{equation}
It is routine to check that the generating function
\begin{equation*}
\widetilde{{B}}(x;z)=\frac{x-1}{(x-1)\cos(z(x-1))-(x+1)\sin(z(x-1))}
\end{equation*}
satisfies~\eqref{EGF}. And this generating function gives
$\widetilde{{B}}(x;0)=1$ and $$\widetilde{{B}}(0;z)=\frac{1}{\cos(z)-\sin(z)}=\sum_{n=0}^\infty \widehat{{B}}_{n}(0)\frac{z^n}{n!}.$$
Hence $\widetilde{{B}}(x;z)=\widehat{B}(x;z)$. This completes the proof.
\end{proof}

\subsection{Alternating descents of type $B$ and left peaks}
\hspace*{\parindent}

A {\it left peak} of $\pi\in\msn$ is an index $i\in[n-1]$ such that $\pi(i-1)<\pi(i)>\pi(i+1)$, where $\pi(0)=0$.
Denote by $\lpk(\pi)$ the number of left peaks in $\pi$. Let $M(n,k)$ be the number of permutations in $\msn$ with $k$ left peaks.
The {\it left peak polynomials} are defined by
$$M_n(x)=\sum_{k=0}^{\lrf{n/2}}M(n,k)x^k.$$
Gessel~\cite[A008971]{Sloane} obtained the following exponential generating function:
\begin{equation}\label{Gessel}
M(x;z)=1+\sum_{n=1}^\infty M_n(x)\frac{z^n}{n!}=\frac{\sqrt{1-x}}{\sqrt{1-x}\cosh(z\sqrt{1-x})-\sinh(z\sqrt{1-x})}.
\end{equation}
By using the theory of enriched $P$-partitions, Petersen~\cite[Proposition~4.15]{Petersen07} obtained that
\begin{equation}\label{Bnxgamma}
B_n(x)=\sum_{k=0}^{\lrf{n/2}}4^iM(n,k)x^k(1+x)^{n-2k},
\end{equation}
which has been extensively studied in recent years, see~\cite{Chow08,Petersen15} and references therein.

The second main result of this paper is given as follows.
\begin{theorem}
We have
  \begin{equation}\label{Bnxgamma01}
\widehat{B}_n(x)=\sum_{k=0}^{\lrf{n/2}}2^kM(n,k)(1+x^2)^k(1+x)^{n-2k}.
\end{equation}
Equivalently, $$\widehat{B}_n(x)=(1+x)^nM_n\left(\frac{2+2x^2}{(1+x)^2}\right).$$
\end{theorem}
\begin{proof}
Substituting $x\rightarrow\frac{2+2x^2}{(1+x)^2}$ and $z\rightarrow (1+x)z$ in~\eqref{Gessel}, we get
$$M\left(\frac{2+2x^2}{(1+x)^2};(1+x)z\right)=\frac{x-1}{(x-1)\cos(z(x-1))-(x+1)\sin(z(x-1))},$$
which yields the desired result.
\end{proof}

\subsection{Derivative polynomials and alternating Eulerian polynomials}
\hspace*{\parindent}

In 1995, Hoffman~\cite{Hoffman95} introduced the derivative polynomials for tangent and secant:
\begin{equation*}\label{derivapoly-1}
\frac{d^n}{d\theta^n}\tan \theta=P_n(\tan \theta),~
\frac{d^n}{d\theta^n}\sec\theta=\sec\theta \cdot Q_n(\tan \theta).
\end{equation*}
By making use of the chain rule it follows that $P_{n+1}(x)=(1+x^2)\frac{\mathrm{d}}{\mathrm{d}x}P_n(x)$ and
\begin{equation}\label{Qnx-recu}
Q_{n+1}(x)=xQ_n(x)+(1+x^2)\frac{\mathrm{d}}{\mathrm{d}x}Q_n(x),
\end{equation}
with the initial conditions $P_0(x)=x$ and $Q_0(x)=1$.
They may be defined by the following generating functions:
\begin{equation*}\label{Pxz-EGF}
P(x;z)=\sum_{n=0}^\infty P_n(x)\frac{z^n}{n!}=\frac{x+\tan z}{1-x\tan z},
\end{equation*}
\begin{equation}\label{Qxz-EGF}
Q(x;z)=\sum_{n=0}^\infty Q_n(x)\frac{z^n}{n!}=\frac{\sec z}{1-x\tan z}.
\end{equation}
Note that
$\frac{\partial}{\partial z}Q(x;z)=P(x;z)Q(x;z)$.
Hence
\begin{equation}\label{Qnx-cov}
Q_{n+1}(x)=\sum_{k=0}^n\binom{n}{k}P_k(x)Q_{n-k}(x).
\end{equation}
The derivative polynomials have been extensively studied
(see~\cite{Cvijovic09,Franssens07,Hetyei08} for instance).
The reader is referred to~\cite{Josuat14,Ma121} for combinatorial interpretations of the derivative polynomials.

In~\cite{Ma16}, Ma and Yeh showed that
\begin{equation}\label{AnxPnx:01}
2^n(1+x^2)\widehat{A}_n(x)=(1-x)^{n+1}P_n\left(\frac{1+x}{1-x}\right).
\end{equation}
The third main result of this paper is given as follows, which also solves~\cite[Conjecture 5.]{Ma16}.
\begin{theorem}\label{thm03}
We have
$$\widehat{B}_n(x)=(1-x)^nQ_n\left(\frac{1+x}{1-x}\right).$$
\end{theorem}
\begin{proof}
Substituting $x\rightarrow\frac{1+x}{1-x}$ and $z\rightarrow (1-x)z$ in~\eqref{Qxz-EGF}, one can immediately get
one can immediately get
$$Q\left(\frac{1+x}{1-x};(1-x)z\right)=\frac{x-1}{(x-1)\cos(z(x-1))-(x+1)\sin(z(x-1))},$$
which yields the desired result.
\end{proof}

Combining~\eqref{Qnx-cov},~\eqref{AnxPnx:01} and Theorem~\ref{thm03}, it is routine to verify the following result.
\begin{corollary}
For $n\geqslant 1$, one has
$$\widehat{B}_{n+1}(x)=(1+x)\widehat{B}_{n}(x)+(1+x^2)\sum_{k=0}^{n-1}\binom{n}{k}2^{n-k}\widehat{B}_{k}(x)\widehat{A}_{n-k}(x).$$
\end{corollary}

\subsection{Alternating gamma expansions and left peak polynomials}
\hspace*{\parindent}

Let $f(x)=\sum_{i=0}^nf_ix^i$. If $f(x)=\sum_{k=0}^{\lfloor n/2\rfloor}\gamma_kx^k(1+x)^{n-2k}$, we call $\{\gamma_k\}_{k=0}^{\lfloor n/2\rfloor}$ the {\it $\gamma$-vector} of $f$. If the $\gamma$-vector of $f$ is nonnegative, then we say that $f$ is $\gamma$-positive. Clearly, that $\gamma$-positivity implies unimodality.
The reader is referred to~\cite{Petersen15} for more details.

Let $\pi\in\msn$.
We say that $\pi$ has no {\it double descents} if there is no
index $i\in [n-2]$ such that $\pi(i)>\pi(i+1)>\pi(i+2)$.
The permutation $\pi$ is called {\it simsun} if for each $k\in [n]$, the
subword of $\pi$ restricted to $[k]$ (in the order
they appear in $\pi$) contains no double descents.
There has been much recent work devoted to simsun permutation and its variations, see~\cite{Chow11,Ma16} for instance.
Let $\rss_n$ be the set of simsun permutations in $\msn$.
Let
 $$S_n(x)=\sum_{\pi\in\rss_n}x^{\des(\pi)}=\sum_{i=0}^{\lrf{n/2}}S(n,i)x^i.$$
It follows from~\cite[Theorem~1]{Chow11} that
$S_{n+1}(x)=(1+nx)S_{n}(x)+x(1-2x)S_{n}'(x)$
for $n\geqslant 1$, with the initial conditions $S_{0}(x)=S_1(x)=1$.
Combining~\cite[Corollary~3.2]{Branden08} and~\cite[Proposition~1]{Ma16}, a gamma expansion of Eulerian polynomials is given as follows:
$$A_{n+1}(x)=\sum_{i=0}^{\lrf{n/2}}2^{i}S(n,i)x^i(1+x)^{n-2i}.$$
Very recently, Lin etal.~\cite{Lin2011} proved that
\begin{equation}\label{Anx-gamma}
\widehat{A}_n(x)=\sum_{k=0}^{\lrf{({n-1})/{2}}}\eta_{n,k}(-2x)^{k}(1+x)^{n-1-2k},
\end{equation}
where $\eta_{n,k}$ are positive integers for all $0\leqslant k\leqslant \lrf{({n-1})/{2}}$.
Let $\eta_n(x)=\sum_{k=0}^{\lrf{({n-1})/{2}}}\eta_{n,k}x^k$. Then $\eta_n(x-1)=S_{n-1}(x)$.
Therefore, the polynomial $\widehat{A}_n(t)$ has $\gamma$-vector alternates in sign.

Recall that the number of left peaks of $\pi\in\msn$ is defined as follows:
$$\lpk(\pi)=\#\left\{i\in [n-1]: \pi(i-1)<\pi(i)>\pi(i+1)\right\},$$
where $\pi(0)=0$.
We can now present the fourth main result of this paper.
\begin{theorem}
For $n\geqslant 0$, we have
\begin{equation}\label{Bnx-gamma}
\widehat{B}_n(x)=\sum_{i=0}^{\lrf{n/2}}\xi_{n,i}(-4x)^{i}(1+x)^{n-2i},
\end{equation}
where
\begin{equation}\label{xi-gamma-recu}
\xi_{n+1,i}=(1+2n-6i)\xi_{n,i}+(n-2i+2)\xi_{n,i-1}-4(i+1)\xi_{n,i+1},
\end{equation}
with the initial conditions $\xi_{0,0}=1$ and $\xi_{0,i}=0$ for $i\neq 0$.
Let $\xi_n(x)=\sum_{i=0}^{\lrf{n/2}}\xi_{n,i}x^i$. Then
\begin{equation}\label{xiQn}
\xi_n(x-1)=x^{\frac{n}{2}}Q_n\left(\frac{1}{\sqrt{x}}\right)=\sum_{\pi\in\msn}(x+1)^{\lpk(\pi)}.
\end{equation}
Equivalently, we have
$$\xi_n(x-2)=\sum_{\pi\in\msn}x^{\lpk(\pi)}.$$ Therefore, the numbers $\xi_{n,i}$ are positive integers for all $0\leqslant i\leqslant \lrf{n/2}$,
and so the polynomial $\widehat{B}_n(x)$ has $\gamma$-vector alternates in sign.
\end{theorem}
\begin{proof}
Let $G_2$ be the grammar given in Theorem~\ref{lemma02}. Setting $a=x+y$ and $b=xy$. Then we have
$D_{G_2}(e)=ea,~D_{G_2}(a)=2a^2-4b,~D_{G_2}(b)=a^3-2ab$.
Consider the grammar $$G_3=\{e\rightarrow ea,~a\rightarrow 2a^2-4b,~b\rightarrow a^3-2ab\}.$$
Note that $D_{G_3}(e)=ea,~D_{G_3}^2(e)=e(3a^2-4b)$ and $D_{G_3}^3(e)=e(11a^3-20ab)$. We proceed by induction.
Assume that
\begin{equation}\label{DG3e}
D_{G_3}^n(e)=e\sum_{i=0}^{\lrf{n/2}}\xi_{n,i}a^{n-2i}(-4b)^{i}.
\end{equation}
It follows that
\begin{align*}
&D_{G_3}^{n+1}(e)=D_{G_3}\left(e\sum_{i=0}^{\lrf{n/2}}\xi_{n,i}a^{n-2i}(-4b)^{i}\right)\\
&=e\sum_{i=0}^{\lrf{n/2}}\xi_{n,i}\left(a^{n-2i+1}(-4b)^{i}+(n-2i)a^{n-2i-1}(2a^2-4b)(-4b)^i+(-4)^ia^{n-2i}ib^{i-1}(a^3-2ab)\right).
\end{align*}
Equating the coefficients of $a^{n-2i+1}(-4b)^{i}$ in both sides of the above expansion, we get~\eqref{xi-gamma-recu}.
Then upon taking $a=1+x$ and $b=x$ in~\eqref{DG3e}, we get~\eqref{Bnx-gamma}. Multiplying both sides of~\eqref{xi-gamma-recu}
by $x^i$ and summing over all $i$, we obtain
\begin{equation}\label{xinx-recu}
\xi_{n+1}(x)=(1+2n+nx)\xi_n(x)-(1+x)(4+2x)\frac{\mathrm{d}}{\mathrm{d}x}\xi_n(x).
\end{equation}

Let $E_n(x)=\xi_n(x^2-1)$. It follows from~\eqref{xinx-recu} that
$$E_{n+1}(x)=(1+n+nx^2)E_n(x)-x(1+x^2)\frac{\mathrm{d}}{\mathrm{d}x}E_n(x).$$
Let $F_n(x)=x^nQ_n\left(\frac{1}{x}\right)$. It follows from~\eqref{Qnx-recu} that
$$F_{n+1}(x)=(1+n+nx^2)F_n(x)-x(1+x^2)\frac{\mathrm{d}}{\mathrm{d}x}F_n(x).$$
Hence $E_n(x)$ satisfies the same
recurrence and initial conditions as $F_n(x)$, so they agree.
In conclusion, we have $$\xi_n(x^2-1)=x^nQ_n\left(\frac{1}{x}\right),$$
which yields the first equality in~\eqref{xiQn}.
The second equality in~\eqref{xiQn} follows from~\cite[Theorem~2]{Ma121}.
\end{proof}

Let $\pi\in\msn$.
We say that $\pi$ is {\it alternating} if $\pi(1)>\pi(2)<\pi(3)>\cdots \pi(n)$.
In other words, $\pi(2i)<\pi({2i+1})$ and $\pi(2i+1)>\pi({2i+2})$ for $0\leqslant i\leqslant \lrf{{(n-1)}/{2}}$.
The {\it secant number} $E_{2n}$ is the number of alternating permutations in $\ms_{2n}$. It is well known that
$$\sec z=\sum_{n=0}^\infty E_{2n}\frac{z^{2n}}{(2n)!}=1+\frac{z^2}{2!}+5\frac{z^4}{4!}+61\frac{z^6}{6!}+\cdots.$$

Below are the polynomials $\xi_n(x)$ for $n\leqslant 5$:
\begin{align*}
&\xi_0(x)=1,~\xi_1(x)=1,~\xi_2(x)=3+x,~\xi_3(x)=11+5x,\\
&\xi_4(x)=57+38x+5x^2,~\xi_5(x)=361+302x+61x^2.
\end{align*}
From~\eqref{xinx-recu}, we see that $\xi_n(-1)=n!$.
Combining~\eqref{Gessel} and~\eqref{xiQn}, we get the following corollary.
\begin{corollary}
We have
\begin{equation}\label{xixz}
\xi(x;z)=\sum_{n=0}^\infty\xi_n(x)\frac{z^n}{n!}=\frac{\sqrt{1+x}}{\sqrt{1+x}\cos(z\sqrt{1+x})-\sin(z\sqrt{1+x})}.
\end{equation}
In particular, $\xi_{n,0}=s_n$ and $\xi_{2n,n}=\xi_{2n-1,n-1}=E_{2n}$, where $s_n$ and $E_{2n}$ are the Springer numbers and secant numbers, respectively.
\end{corollary}
\subsection*{Acknowledgements}
This work is supported by NSFC 12071063.





\begin{thebibliography}{14}

\bibitem{Arnold}
V.I. Arnold, \textit{The calculus of snakes and the combinatorics of Bernoulli, Euler and Springer numbers of Coxeter groups}, Russian Math. Surveys, \textbf{47} (1) (1992), 1--51.


\bibitem{Branden08}
P. Br\"and\'{e}n, \textit{Actions on permutations and unimodality of descent polynomials}, European J. Combin., \textbf{29} (2008), 514--531.

\bibitem{Brenti94}
F. Brenti, \textit{$q$-Eulerian polynomials arising from Coxeter groups}, European J. Combin.,
\textbf{15} (1994), 417--441.

\bibitem{Chebikin08}
D. Chebikin, \textit{Variations on descents and inversions in permutations}, Electron. J.
Combin., \textbf{15} (2008), \#R132.

\bibitem{Chen93}
W.Y.C. Chen, \textit{Context-free grammars, differential operators and formal
power series}, Theoret. Comput. Sci., \textbf{117} (1993), 113--129.

\bibitem{Chen17}
W.Y.C. Chen, A.M. Fu, \textit{Context-free grammars for permutations and increasing trees}, Adv. in Appl. Math., \textbf{82} (2017), 58--82.

\bibitem{Cvijovic09}
D. Cvijovi\'c, \textit{Derivative polynomials and closed-form higher derivative formulae},
Appl. Math. Comput., \textbf{215} (2009), 3002--3006.

\bibitem{Chow08}
C.-O. Chow, \textit{On certain combinatorial expansions of the Eulerian polynomials}, Adv. in Appl. Math., \textbf{41} (2008), 133--157.


\bibitem{Chow11}
C-O. Chow, W. C. Shiu, \textit{Counting simsun permutations by descents}, Ann. Comb., \textbf{15} (2011), 625--635.

\bibitem{Dumont96}
D. Dumont, \textit{Grammaires de William Chen et d\'erivations dans les arbres et
arborescences}, S\'em. Lothar. Combin., \textbf{37}, Art. B37a (1996), 1--21.


\bibitem{Franssens07}
G.R. Franssens, \textit{Functions with derivatives given by polynomials in the function itself or a related function}, Anal. Math., \textbf{33} (2007), 17--36.

\bibitem{Gessel14}
I.M. Gessel, Y. Zhuang, \textit{Counting Permutations by alternating descents}, Electron. J.
Combin., \textbf{21}(4) (2014), \#P4.23.

\bibitem{Hetyei08}
G. Hetyei, \textit{Tchebyshev triangulations of stable simplicial complexes},
J. Combin. Theory Ser. A, \textbf{115} (2008), 569--592.


\bibitem{Hoffman95}
M.E. Hoffman, \textit{Derivative polynomials for tangent and secant}, Amer. Math. Monthly, \textbf{102} (1995), 23--30.


\bibitem{Josuat14}
M. Josuat-Verg\`{e}s,  \textit{Enumeration of snakes and cycle-alternating permutations}, Australas. J.
Combin., \textbf{60} (2014), 279--305.

\bibitem{Lin2011}
Z. Lin, S.-M. Ma, D.G.L. Wang, L. Wang, \textit{Positivity and divisibility of alternating descent polynomials}, \arxiv{arXiv:2011.02685}.

\bibitem{Ma121}
S.-M. Ma, \textit{Derivative polynomials and enumeration of permutations by number of interior and left peaks}, Discrete Math., \textbf{312} (2012), 405--412.


\bibitem{Ma131}
S.-M. Ma, \textit{Some combinatorial arrays generated by context-free grammars}, European J. Combin.,
\textbf{34} (2013), 1081--1091.

\bibitem{Ma16}
S.-M. Ma, Y.-N. Yeh, \textit{Enumeration of permutations by number of
alternating descents}, Discrete Math., \textbf{339} (2016), 1362--1367.

\bibitem{Ma1601}
S.-M. Ma, Y.-N. Yeh, \textit{The peak statistics on simsun permutations}, Electron. J. Combin., \textbf{23}(2) (2016), \#P2.14.


\bibitem{Petersen07}
T.K. Petersen, \textit{Enriched P-partitions and peak algebras}, Adv. Math., \textbf{209} (2007), 561--610.

\bibitem{Petersen15}
T.K. Petersen, \textit{Eulerian Numbers}. Birkh\"auser/Springer, New York, 2015.


\bibitem{Remmel12}
J.B. Remmel, \textit{Generating functions for alternating descents and
alternating major index}, Ann. Comb., \textbf{16} (2012), 625--650.

\bibitem{Zeng16}
H. Shin and J. Zeng, \textit{Symmetric unimodal expansions of excedances in colored permutations}, European
J. Combin., \textbf{52} (2016), 174--196.

\bibitem{Sloane}
N.J.A. Sloane, The On-Line Encyclopedia of Integer Sequences,
published electronically at
http://oeis.org, 2010.

\bibitem{Springer71}
T.A. Springer, \textit{Remarks on a combinatorial problem}, Nieuw Arch. Wisk., \textbf{19} (1971), 30--36.
\end{thebibliography}
\end{document}